\documentclass[10pt,a4paper]{elsarticle} 
\usepackage{enumerate}
\usepackage{amsmath,amsthm}
\usepackage{amsfonts}
\usepackage{amssymb}
\usepackage{algorithm, algorithmic}
\usepackage{hyperref}
\usepackage{xcolor}
\usepackage{graphicx}
\usepackage{arydshln}
\usepackage{subcaption}
\usepackage{tikz}
\usetikzlibrary{matrix,positioning,decorations.pathreplacing}
\usepackage[width=15cm, left=2cm, right=2cm, vscale=0.8, nohead]{geometry} 
\numberwithin{equation}{section}
\newtheorem{theorem}{Theorem}[section]
\newtheorem{proposition}[theorem]{Proposition}
\newtheorem{lemma}[theorem]{Lemma}
\newtheorem{corollary}[theorem]{Corollary}
\theoremstyle{definition}
\newtheorem{definition}[theorem]{Definition}

\newtheorem{example}[theorem]{Example}

\journal{}
\begin{document}
\begin{frontmatter}
\title{Some results about the Equiangular Algorithm}
\author[sad]{D. Sadeghi\corref{cor1}}
\author[sad]{Azim Rivaz}
\ead{arivaz@uk.ac.ir}
\address[sad]{Department of Mathematics, Shahid Bahonar University of Kerman, Kerman 76169-14111, IRAN}
\cortext[cor1]{Corresponding author. E-mail: $\mathtt{dl.sadeghi@math.uk.ac.ir}$.}
\begin{abstract}
\textit{Equiangular Algorithm} generates a set of equiangular normalized vectors with given angle $\theta$ using a set of linearly independence vectors in a real inner product space, which span the same subspaces. The outcome of EA on column vectors of a matrix $A$ provides a matrix decomposition $A=SR$, where $S$ is called \textit{Equiangular Matrix} which has equiangular column vectors.
\\In this paper we discuss some properties of equiangular matrices. The inverse and eigenvalue problems of these matrices are studied. Also we derive some canonical forms of some matrices based on equiangular ones.
\end{abstract}
\end{frontmatter}
\section{Introduction} \label{intro}
\textit{Equiangular Algorithm} (EA) \cite{Rivaz} is a process like Gram-Schmidt algorithm \cite{Horn,Solivérez}, that takes a set of linearly independent vectors $\{a_1,\cdots,a_m\}\subset \mathbb{R}^n~(m\leq n)$, then produces a set of normalized equiangular vectors $\{s_1,\cdots,s_m\}$ with the angle $\theta\in(0,\arccos(\frac{-1}{m-1}))$. So $s_i^T s_j=\cos\theta= \alpha$ if $i\neq j$, and $\Vert s_i\Vert=1$ for $i=1,\ldots,m$. We then showed in \cite{Rivaz} that any $n\times m$ full rank matrix $A$ is factorized as $A=SR$, where $S$ is an equiangular matrix $n\times m$ and $R$ is an $m\times m$ upper triangular matrix (SR decomposition). We denote the set of all full-rank equiangular matrices of size $n\times m$ with $\cos\theta=\alpha$ by $\text{EM}^{n\times m}_\alpha$ and the nonsingular equiangular square matrices by $\text{EM}^n_\alpha$. Also $G_\alpha=S^T S$ is a positive definite matrix. In EA the $(k+1)$th equiangular vector $s_{k+1}$ is obtained as follows
\begin{equation}
\quad s_{k+1}=\frac{v_{k+1}}{\Vert v_{k+1}\Vert}~~,~~v_{k+1}=q_{k+1}+\sqrt{ \tfrac{k}{(\sec\theta-1)(\sec\theta+k)}}\frac{\sum _{i=1}^{k}{s_i}}{\Vert\sum _{i=1}^{k}{s_i}\Vert}, \label{s_k+1}
\end{equation}
where the vector $q_{k+1}$ is a normalized orthogonal vector to $s_i$'s which of $span\langle a_1,\ldots,a_{k+1}\rangle$  $(i=1,\ldots,k)$. Since the size of the right hand vector in \eqref{s_k+1} is $\sqrt{ \tfrac{k}{(\sec\theta-1)(\sec\theta+k)}}$, then if $\theta\rightarrow\pi/2$, so this coefficient tends to zero and $v_{k+1}$ converges to the $q_{k+1}$.
\\This paper is organized as follows. In section \ref{inveig}, we show that the inverse of equiangular matrices can be computed with order of $O(n^2)$. Also we provide a bound for the eigenvalues of an equiangular matrix. In section \ref{schur}, we introduce some matrix factorizations based on Schur decomposition, then provide its application in related to the roots of a polynomial. In section \ref{doubequi} we study the special set of normal matrices named ``doubly equiangular matrices" which are equiangular as column-wise and row-wise. Finally, in section \ref{equiframe}, we introduce the set of equiangular tight frame of size $n+1$ in $\mathbb{R}^n$. Then the existence of equiangular vectors with angles greater than $\pi/2$ will be discussed using equiangular tight frames. \\
Throughout this paper, $\Vert X\Vert$ denotes the 2-norm of matrix $X$, $\Vert x\Vert$ denotes the Euclidean norm of vector $x$. Also $\alpha=\cos\theta$, $e=[1,\ldots ,1]^T\in\mathbb{R}^n$ and $\{e_1,e_2,\ldots ,e_n\}$ is the standard orthogonal basis of $\mathbb{R}^n$. Moreover, All matrices in this paper are real. For simplicity, we denote the transpose of the inverse of a non-singular matrix $A$ as $A^{-T}$. Indeed any set of equiangular vectors in $\mathbb{R}^n$ can be considered as a set of equiangular lines (ELs), but not inversely, in general. The discussion of ELs is of interest for about sixty years of investigation. In 1973, Lemmens and Seidel \cite{Lemmens} made a comprehensive study of real equiangular line sets which is still today a fundamental piece of work. In this paper we turn our attention on equiangular vectors.
\section{Inverse and eigenvalue problems} \label{inveig}
In this section we discuss the inverse and eigenvalue problems of equiangular matrices. We need to introduce the so-called $\mathit{Gram~matrix}$ \cite{Godsil,Horn}. 
\begin{definition}
If $A=[a_1,\ldots,a_m]$, then the matrix $G=A^T A$, is said the $\mathit{Gram~matrix}$ of $A$.\qquad\qquad\qquad\quad $\lozenge$ \label{defG}
\end{definition}
Suppose that $S=[s_1,\ldots,s_n]\in \text{EM}^n_\alpha$ then we define $G_\alpha$ as
\begin{equation}
S^{T}S={\footnotesize \left[\begin{array}{cccc}1&\alpha &\cdots &\alpha\\\alpha &1&\cdots &\alpha\\\vdots &\vdots &\ddots &\vdots\\\alpha &\alpha &\cdots &1\end{array}\right]} \label{I},
\end{equation}
which is the Gram matrix of $S$. Since for $x\neq 0$ we have $x^TG_\alpha x=\Vert Sx\Vert^2> 0$, then $G_\alpha$ is positive definite. $G_\alpha$ can be rewritten as $I+\alpha\mathcal{S}$, where $\mathcal{S}$ has zeros on its main diagonal and ones on all off-diagonals and is a special case of the $\mathit{Seidel~matrix}$. 
\begin{proposition}
If $S\in EM_\alpha^n$, then $S^{-1}=\beta G_{\alpha^\prime}S^T$ and can be computed with $O(n^2)$ arithmetic operations where \label{invS}
\begin{equation}
\beta=\frac{1+(n-2)\alpha}{(1-\alpha)\big(1+(n-1)\alpha\big)}\qquad,\qquad {\alpha^{\prime}}=\frac{-\alpha}{1+(n-2)\alpha}~. \label{alphbet}
\end{equation}  
\end{proposition}
\begin{proof}
Since $S^{-1}S^{-T}=G_\alpha^{-1}=\beta G_{\alpha^\prime}$, then the result is obvious. If $[S]_{ij}=s_{ij}$ then $[S^{-1}]_{ij}=\beta(s_{ji}+\alpha^{\prime} \cdot\sum_{k\not =i}s_{jk})$. Therefore $S^{-1}$ can be computed with $O(n^2)$ arithmetic operations.
\end{proof}
If $\alpha$ be zero, then $\beta=1$ and $G_{\alpha^{\prime}}=G_0=I$ so $S^{-1}=S^T$ which shows that the result is true in case of orthogonality.
\begin{corollary}
Let $S=[s_1,\ldots,s_n]\in \text{EM}^n_\alpha$ by $0<\alpha<1$, then the rows of $S^{-1}$, $(s_{i}^{\prime})$'s are equiangular with $\arccos~\alpha^{\prime}=\theta ^{\prime}>\pi/2$, where $\Vert s_{i}^{\prime}\Vert=\sqrt{\beta}$. Also the cosine of the angle between $s_i$ and $s_{i}^{\prime}$ is $1/\sqrt{\beta}$, where $\alpha^{\prime},\beta$ are defined in \eqref{alphbet}.          \label{Tinv}
\end{corollary}
\begin{proof}
If $\tau_i$ be the angle between vectors $s_i, s_{i}^{\prime T}$, then $1=s_{i}^{\prime}\cdot s_i=\Vert s_{i}^{\prime}\Vert\cos\tau_i$, so $\cos\tau_{i}=1/\Vert s_{i}^{\prime}\Vert$ and $\tau_i\in(0,\pi/2)$. Two vectors $s_i$ and $s_{i}^{\prime T}$ make the angles $\theta$ and $\pi/2$ with all vectors of the set $\tilde{S}_{i}=\{s_j\}_{j\neq i}$, respectively. Therefore we introduce the subspace $\mathbf{P}_{n,i}$ with respect to $i$ similar to $\mathbf{P}_k$ in \cite{Rivaz} (eq. (2.6)) in which all vectors has the same angle with the vectors of $\tilde{S}_i$. So $s_i,s_{i}^{\prime T}\in \mathbf{P}_{n,i}$. On the other hand $z_i=\sum_{j\neq i}s_{j}\in\mathbf{P} _{n,i}$. As shown in \cite{Rivaz}, $\mathbf{P}_{n,i}$ is a plane, hence $z_i \in\text{span}\langle s_i,s_{i}^{\prime T}\rangle$. Clearly $s_{i}^{~\prime T}\perp z_i$. If $\varphi_i$ be the angle between $s_i,z_i$, then from Pythagoras Theorem, $1=\cos^{2}\varphi_i+\cos^{2}\tau_i =(n-1)\alpha^2/\big(1+ (n-2)\alpha\big)+1/\Vert s_{i}^{\prime}\Vert^2$, thus
\begin{equation} 
\Vert s_{i}^{~\prime}\Vert =\sqrt{\frac{1+(n-2)\alpha}{(1-\alpha)\left(1+(n-1)\alpha\right)}}= \sqrt\beta~,\quad i=1,\ldots ,n\label{ninv}
\end{equation}
Taking inner product of rows in two sides of $S^{-1}=\beta G_{\alpha^{\prime}}S^T$ gives
\begin{align} 
s_{i}^{\prime}s_{j}^{\prime T} 
&=\beta^{2}[\alpha^{\prime}\cdots \substack{i\text{th}\\1\\~}\cdots \alpha^{\prime}]~
G_\alpha~ 
[\alpha^{\prime}\cdots \substack{j\text{th}\\1\\~}\cdots \alpha^{\prime}]^T\nonumber\\& =\beta^{2}\big[\left(\alpha +\alpha^{\prime}+(n-2)\alpha \alpha^{\prime}\right)[1\cdots \substack{i\text{th}\\0\\~}\cdots 1]+\left(1+(n-1)\alpha\alpha ^{\prime}\right)[0\cdots\substack{i\text{th} \\1\\~}\cdots0]\big][\alpha^{\prime}\cdots\substack{j\text{th} \\1\\~}\cdots\alpha ^{\prime}]^T\nonumber\\&=\beta^{2} \big[\left(\alpha+\alpha ^{\prime}+(n-2)\alpha\alpha^{\prime}\right)((n-2)\alpha^{\prime}+1)+ \left(1+(n-1)\alpha\alpha^{\prime} \right)\alpha^{\prime}\big]=\frac{-\alpha}{(1-\alpha)\left(1+(n-1)\alpha\right)}~.\label{sipsjp}
\end{align}
Then $\cos\theta^\prime=\dfrac{s_{i}^{\prime} s_{j}^{\prime T}}{\Vert s_{i}^{\prime}\Vert\Vert s_{j}^{\prime T}\Vert}= \dfrac{-\alpha}{1+(n-2)\alpha}=\alpha^{\prime}$. So $\alpha^{\prime}<0$ and $\theta^\prime>\pi/2$.
\end{proof}
From \eqref{ninv} matrix $\beta^{-1/2}S^{-1}$ is row-wise equiangular, so $\beta^{-1/2}S^{-T}\in \text{EM}^n_{\alpha^{\prime}}$. for $n=2$, $\alpha^{\prime}=-\alpha$ and $\theta^{'}=\pi -\theta$ (Figure \ref{fig4}).
\begin{figure}
\centering
\includegraphics[scale=.22]{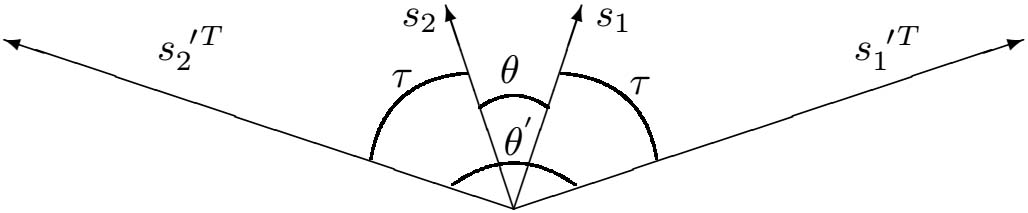}
\caption{\scriptsize{Inverse of the matrix $[s_1,s_2]$ is $[s_{1}^{~\prime},s_{2}^{~\prime}]^{T}, \text{where}~\theta +\tau =\frac{\pi}{2}$ and $\theta+\theta^{'} =\pi$.}}
\label{fig4}
\end{figure}
\begin{corollary}
If $\alpha^{\prime},\beta$ are defined as \eqref{alphbet}, then $G_\alpha G_{\alpha^{\prime}}=(1/\beta)I_n$~. \label{invG}
\end{corollary}
\begin{proof}
It is obvious from Proposition \ref{invS}.
\end{proof}
We give some examples of computing the inverse of equiangular matrices using \ref{invS}.
\begin{example}
For the $4\times 4$ Hilbert matrix \textit{H} EA for $\theta=\pi/3$ gives $H=SR$ where\\ $S=\left[\footnotesize{\begin{array}{cccc}0.8381&-0.0336&~~0.3939 &~~0.2788\\0.4191& ~~0.5921&-0.2572&~~0.4381\\0.2794&~~0.5977& ~~0.4062&-0.3031\\0.2095&~~0.5396&~~0.7834& ~~0.7991\end{array}}\right]$ and $R=\left[\footnotesize{\begin{array}{cccc}1.1932&0.6021&0.3998&0.2980\\0~\qquad& 0.1369&0.1426&0.1318\\0~ \qquad&0~\qquad&0.0076& 0.0117\\0~\qquad&0~\qquad&0~\qquad&0.0002 \end{array}}\right]$. $\alpha={\alpha^{\prime}}=1/2$ and $\beta=3/2$, then $S^{-1}=1.5~G_{0.5}\left[\footnotesize {\begin{array}{cccc}~~0.8381&~~0.4191&~~0.2794 &0.2095\\-0.0336&~~0.5921&~~0.5977&0.5396\\ ~~0.3939&-0.2572&~~0.4062&0.7834 \\~~0.2788&~~0.4381&-0.3031&0.7991 \end{array}}\right]$.\qquad\qquad\qquad\qquad\qquad \qquad\qquad\quad $\lozenge$
\end{example}
\begin{example}
For the Identity matrix $I_4$, EA for $\theta=\pi/4$ gives\\ $S=\left[\footnotesize{\begin{array}{cccc}1&0.7071&0.7071&0.7071\\0&0.7071 &0.2929&0.2929\\0&0~\qquad&0.6436& 0.1885\\0&0~\qquad&0~\qquad&0.6154 \end{array}}\right]$ and $R=S^{-1}=\left[\footnotesize{\begin{array}{cccc} 1&-1\quad\qquad&-0.6436~~&-0.4760\\ 0&1.4142&-0.6436~~&-0.4760\\0&0~\qquad &1.5538&-0.4760\\0&0~\qquad&0~\qquad &~~1.6250 \end{array}}\right]$. By rounding we have $S\in \text{EM}^4_{0.7071}$ and $(0.6154)R^T\in \text{EM}^4_{-0.2929}$. Note that there is only one upper triangular equiangular matrix with positive entries with respect to a scaler $\alpha$.\qquad\qquad\qquad\qquad\qquad \qquad\qquad\qquad\qquad \qquad\qquad\qquad\qquad\qquad $\lozenge$
\end{example}
\begin{example}
For the orthogonal matrix $Q=\left[\footnotesize{\begin{array}{ccc}~~3/7&-2/7&~~6/7\\~~6/7&~~3/7 &-2/7\\-2/7&~~6/7&~~3/7\end{array}} \right]$ EA for $\theta=\pi/3$ gives $Q=SR$ where $S=\left[\footnotesize{\begin{array}{ccc}~~0.4286&-0.0332&0.8317\\~~0.8571 &~~0.7997&0.3190\\-0.2857&~~0.5995& 0.4545\end{array}}\right]$ and $R=\left[\footnotesize{\begin{array}{ccc}1& -0.5774&-0.4082\\0&~~1.1547&-0.4082\\ 0&0~~~~~&~~1.2247\end{array}}\right]$. As regards to \eqref{alphbet}, $\beta^{-1/2}R$ is row-equiangular: $RR^T =S^{-1}S^{-T}=G_ \alpha ^{-1}=\beta G_{\alpha^{\prime}}$, with $\alpha=1/2$. \label{Dorth} \qquad\qquad\qquad\qquad\qquad \qquad\qquad\qquad\qquad\qquad\quad $\lozenge$
\end{example}
Now we discuss the eigenvalues of Equiangular matrices. Actually we present lower and upper bounds for the eigenvalues of an Equiangular matrix $S$ relative to the eigenvalues of its corresponding matrix $G_\alpha$~.
\begin{lemma}
Suppose that $G_\alpha$ is the Gram matrix of a given matrix $S\in \text{EM}^n_\alpha$. The eigenvalues of $G_\alpha$ are $1-\alpha$ and $1+(n-1)\alpha$. \label{Leig}
\end{lemma}
\begin{proof}
Since $G_\alpha e=\alpha e e^{T}e+(1-\alpha)I_n e=(1+(n-1)\alpha)e$, then $(1+(n-1)\alpha,e^{T})$ is an eigenpair of $G_\alpha$. Take $x=[x_1,\ldots,x_n]^T$ with $x_1+\ldots+x_n=0$, so
$G_\alpha x=\alpha ee^{T} x+(1-\alpha) I_n x=(1-\alpha)x$. Thus $1-\alpha$ is the second eigenvalue of $G_\alpha$ with the algebraic multiplicity $n-1$. Therefore $\sigma(G_\alpha)=\{1-\alpha,1+(n-1)\alpha\}$. 
\end{proof}
As noted in \cite{Rivaz} (Proposition 1.2), a set of equiangular vectors $s_1,\ldots,s_n$ can be embedded into the positive coordinate axes so that $n-1$ entries of all of them are the same as $t\in(0,\tfrac{1}{\sqrt{n}})$ and the last one is $s=\sqrt{1-(n-1)t^2}$. Since $\alpha=2ts+(n-2)t^2$, then it can be shown that
\begin{equation}
s=\tfrac{\sqrt{(1-\alpha)(n-1)(n-2)+n\pm 2(n-1)\sqrt{(1-\alpha)(1+(n-1)\alpha)}}}{n} \quad ,\quad t=\tfrac{\sqrt{\alpha n+2(1-\alpha\mp\sqrt{(1-\alpha)(1+(n-1)\alpha)})}}{n}, \label{s,t1}
\end{equation}
so that the plus sign in the formula of $s$ must be selected and vise versa in $t$. For this reason $s\in(\frac{1}{\sqrt{n}},1)$ and $S$ with these vectors is positive definite with positive eigenvalues. $S$ can be rewritten as $sG_{t/s}$. From Lemma \ref{Leig} $\sigma(sG_{t/s})=\{ s-t,s+(n-1)t\}$. We define $\bar{S}_\alpha=sG_{t/s}$. It is notable that $\bar{S}_\alpha$ is the unique principal square root of $G_\alpha$, i.e., $\bar{S}_\alpha={G_ \alpha}^{1/2}$ \cite{Higham}. Since $\bar{S}_\alpha,S\in \text{EM}^n_\alpha$, then there exists an orthogonal matrix $Q$ so that $S=Q\bar{S}_\alpha$. Since $\bar{S}_\alpha$ is positive definite, $S$ is nonsingular and $\bar{S}_\alpha =(S^T S)^{1/2}$, then the last equality is the ``polar decomposition" of $S$ \cite{Higham1}. This equality can be interpreted as a transformation of the orthogonal matrices to the equiangular ones and vise versa ($Q=S\bar{S}_\alpha^{-1}$).\\
If $(\mu,x)$ is an eigenpair of $\bar{S}_\alpha$, then $(\mu^2,x)$ is an eigenpair of $G_\alpha$. So $\sigma(\bar{S}_\alpha)=\{\sqrt{1-\alpha},\sqrt{1+(n-1)\alpha}\}=\{s-t,s+(n-1)t\}$ and $s,t$ are as follows
\begin{equation}
s=\frac{\sqrt{1+(n-1)\alpha}+(n-1)\sqrt{1- \alpha}}{n}\quad ,\quad t=\frac{\sqrt{1+(n-1)\alpha}- \sqrt{1-\alpha}}{n}\cdot \label{s,t2}
\end{equation}
\begin{example}
For the Gram matrix $G_{1/2}={\footnotesize\left[\begin{array}{ccc}1&1/2&1/2\\1/2&1&1/2\\1/2 &1/2&1\end{array}\right]}$, the rounded form of the square roots of $G_{1/2}$ is dependent on scalars $s,t$ as follows
\begin{align}
\sqrt{G_{1/2}}&={\footnotesize \left[\begin{array}{ccc}0.9428&0.2357&0.2357\\0.2357& 0.9428&0.2357\\0.2357&0.2357 &0.9428\end{array}\right]},~ \text{for}~~0<t<\tfrac{1}{\sqrt{n}}~,\nonumber
\\\sqrt{G_{1/2}}&={\footnotesize \left[\begin{array}{ccc}0&0.7071&0.7071\\0.7071 &0&0.7071\\0.7071&0.7071&0 \end{array}\right]},~\text{for}~~ \tfrac{1}{\sqrt{n}}<t\leq\tfrac{1}{\sqrt{n-1}}.
\end{align}
Since $\sigma(\bar{S}_{1/2})=\{s-t,s+(n-1)t\}$, then the eigenvalues of first matrix are $0.7071,1.4142$ and those of the second one are $-0.7071,1.4142$. Since $\bar{S}_{1/2}$ must be positive definite and $s>0$, then the first one is $\bar{S}_{1/2}$. \label{psquare}
\end{example}
In the next theorem the lower and upper bounds for eigenvalues of equiangular matrices are presented.
\begin{theorem}
Let $(\lambda,x)$ be an eigenpair of an equiangular matrix $S\in \text{EM}^n_\alpha$ with $\Vert x\Vert =1$ then the following bounds hold:
\begin{equation}
\lambda _{\text{min}}(\bar{S}_\alpha)\leq\vert \lambda \vert\leq \lambda_{\text{max}} (\bar{S}_\alpha),
\end{equation}
where $\lambda_{min}$ and $\lambda_{max}$ stand for the minimum and maximum eigenvalues of $\bar{S}_\alpha$, respectively.
\end{theorem}
\begin{proof}
Since $Sx=\lambda x$, and $x^{*}S^{T}=\lambda^{*}x^{*}$, then $x^{*}S^{T}Sx=\lambda \lambda^{*} x^{*}x=\vert\lambda\vert^{2}\Vert x\Vert^{2}=\vert\lambda\vert^{2}$. On the other hand
\begin{align}
x^{*} S^{T}Sx &=x^{*}G_\alpha x=x^ {*}(\alpha ee^T+(1-\alpha)I)x=\alpha (\sum_{i=1} ^{n} x_i) (\sum_{i=1}^{n}x^{*}_i)+(1-\alpha)\Vert x\Vert^{2}\nonumber\\&=\alpha\big|\sum_{i=1}^{n}x_i\big|^{2}+1-\alpha=\alpha\vert e^Tx \vert^{2}+1-\alpha ~,
\end{align}
where is the vector of ones. Taking these equalities together implies 
\begin{equation}
\vert\lambda\vert =\sqrt{\alpha ~\vert e^T x \vert^{2}+1-\alpha}~.\label{eig}
\end{equation}
Equation \eqref{eig} describes the relationship between eigenvalues and eigenvectors of $S$. Maximum of $\vert e^T x\vert$ is attained if $x=\pm\frac{e}{\Vert e\Vert}$ which is an eigenvector of $\bar{S}_\alpha$. So $\vert\lambda\vert\leq\sqrt{\alpha n+1-\alpha}=\sqrt{1+(n-1)\alpha} =\lambda_{\text{max}}\bar{S}_\alpha$ if $0<\alpha<1$.
Likewise, minimum of $\vert e^T x\vert$ is attained if $e^{T}x=0$ which in this case $x\in\ker(e^{T}x)$. Also $\vert\lambda\vert\geq \sqrt{1-\alpha} =\lambda_{\text{min}}\bar{S}_\alpha$ if $0<\alpha<1$. The same result holds for case of $\alpha<0$. 
\end{proof} 
\begin{proposition}
The condition number of any matrix $S\in \text{EM}^n_\alpha$ relative to $2$-norm is equal to $\sqrt{1+\dfrac{n\alpha}{1-\alpha}}$ if $0<\alpha<1$ and $\sqrt{1+\dfrac{n\vert\alpha\vert}{1-(n-1)\vert\alpha\vert}}$ if $\frac{-1}{n-1}<\alpha<0$.
\end{proposition}
\begin{proof}
Since $\Vert S\Vert_2 =\sqrt{\lambda_{\text{max}}(S^T S)}$ and $\Vert S^{-1}\Vert_2 =\sqrt{\lambda_{\text{max}}(SS^T)^{-1}}=\dfrac{1}{\sqrt{\lambda_{\text{min}}(S^T S)}}$, then $\kappa_2 (S)=\sqrt{\dfrac{\lambda_{\text{max}}(S^T S)}{\lambda_{\text{min}}(S^T S)}}$. So for any case of $\alpha$ the result is obvious.
\end{proof}
Note that $S$ converges to ill-conditioning as $\alpha\rightarrow 1$ or $\alpha\rightarrow-1/(n-1)$.
\section{Some generalizations of the Schur form} \label{schur}
\begin{proposition}
For $A\in\mathbb{R}^{n\times n}$ and $\alpha\in(0,1)$, there exists $S\in \text{EM}^n_\alpha$ so that $S^{-1}AS=T$ is a block upper triangular, with $1\times 1$ and $2\times 2$ blocks on its diagonal. The eigenvalues of $A$ are the eigenvalues of diagonal blocks of $T$. The $1\times 1$ blocks correspond to real eigenvalues, and the $2\times 2$ blocks to pairs of complex conjugate eigenvalues.
\label{STiS}
\end{proposition}
\begin{proof}
From the Real Schur form $A=Q\Lambda Q^T$, where $Q$ is an orthogonal matrix and $\Lambda$ is a block upper triangular. From SR decomposition \cite{Rivaz} we have $Q=SR$, where $S\in \text{EM}^n_\alpha$ and $R$ is upper triangular. Then $A=Q\Lambda Q^T=SR\Lambda R^{-1} S^{-1}=STS^{-1}$. Therefore $T=R\Lambda R^{-1}$ is a block upper triangular, whose blocks are conformable with those of $\Lambda$.
\end{proof}
We want to know which matrices have equiangular eigenvectors. If for some special matrix $A$, there is a matrix $S\in \text{EM}^n_\alpha$ so that $S^{-1}AS=T=\text{diag}(t_1,\ldots,t_n)$, then $A$ has $n$ equiangular eigenvectors. We can provide a test to check this for a matrix. First we intoduce the upper triangular equiangular matrices.
\begin{lemma}
There is a unique triangular equiangular matrix in terms of $\alpha$ as follows
\begin{equation}
\hat{S}=\left[\footnotesize{\begin{array}{ccccc}1&\alpha &\alpha &\cdots &\alpha\\0&\sqrt{1-\alpha^2} &\frac{\alpha(1-\alpha)}{\sqrt{1-\alpha^2}}&\cdots & \frac{\alpha(1-\alpha)}{\sqrt{1-\alpha^2}}\\0&0&\ddots\qquad &~&\vdots\\\vdots &\vdots &\ddots &~&~ \\0&0&\cdots &~&~\end{array}}\right] \label{shat}
\end{equation}
which is obtained from the SR decomposition of the Identity matrix and its entries satisfies to the following
\begin{itemize}
\item $\hat{s}_{11}=1,$
\item $\hat{s}_{i(i+1)}=\cdots =\hat{s}_{in}=\hat{s}_{ii}-\dfrac{1-\alpha}{\hat{s}_{ii}},$
\item $\hat{s}_{ii}^2=1-(\hat{s}_{1i}^2+ \hat{s}_{2i}^2+\cdots+\hat{s}_{(i-1)i}^2).$
\end{itemize}
\end{lemma}
\begin{proof}
Clearly $\hat{s}_{11}=1$. Since $\hat{s}_1^T\hat{s}_i=\alpha$, then $\hat{s}_{12}=\cdots=\hat{s}_{1n} =\dfrac{\alpha}{\hat{s}_{11}}$. Also $\hat{s}_i^T\hat{s}_j=\alpha$ $(i<j)$, therefore by induction
\begin{equation}
\hat{s}_{i(i+1)}=\hat{s}_{i(i+2)}=\cdots =\hat{s}_{in}=\frac{\alpha -(\hat{s}_{12}^2+\hat{s}_{23}+\cdots +\hat{s}_{(i-1)i}^2)}{\hat{s}_{ii}}\cdot \label{sij}
\end{equation}
On the other hand $\Vert \hat{s}_i\Vert=1$ then $(\hat{s}_{1i}^2+\cdots +\hat{s}_{(i-1)i}^2)=1-\hat{s}_{ii}^2$. Also from \eqref{shat}, $\hat{s}_{i(i+1)}=\hat{s}_{ij}$ $i=1,\ldots,n-2,~j=i+2,\ldots,n$. So the equation \eqref{sij} can be written more simpler as follows
\begin{align}
\hat{s}_{i(i+1)}=\hat{s}_{i(i+2)}=\cdots =\hat{s}_{in}&=\frac{\alpha -(\hat{s}_{1i}^2+\hat{s}_{2i}+\cdots +\hat{s}_{(i-1)i}^2)}{\hat{s}_{ii}}=\frac{\alpha -(1-\hat{s}_{ii}^2)}{\hat{s}_{ii}}\nonumber\\&=\hat{s} _{ii}-\frac{1-\alpha}{\hat{s}_{ii}},
\end{align}
where this equation is used for computation of $i$th row.
\end{proof}
\begin{theorem}
For nonsingular and nonsymmetric matrix $A\in\mathbb{R}^{n\times n}$ with the Schur form $A=QTQ^T$ and by the assumption $[T]_{ij}=t_{ij}$
\begin{enumerate}
\item If for a variable $i$ that $t_{ii}\neq t_{(i+1)(i+1)}$, there is a $\alpha\in (0,1)$ satisfies $\frac{\hat{s}_{i(i+1)}}{\hat{s}_{(i+1)(i+1)}}=\frac{t_{i(i+1)}}{t_{(i+1)(i+1)}-t_{ii}}$, whit $\hat{s}_{ij}\in\hat{S}\in \text{EM}^n_\alpha$. Then the column vectors of $S=Q\hat{S}$ form the eigenvectors of $A$, if $\hat{S}$ satisfies to the equality $T\hat{S}=\hat{S}\text{diag}(T)$.
\item If $t_{11}=\ldots =t_{nn}$, then $A=t_{11}I$ have an equiangular eigenspaces if $T$ is diagonal.
\end{enumerate}
Otherwise $A$ has no equiangular eigenvectors.
\end{theorem}
\begin{proof}
If for a $i$ that $t_{ii}\neq t_{(i+1)(i+1)}$ the entry $(i,i+1)$ in two side of $T\hat{S}=\hat{S}\text{diag}(T)$ is considered, so from the first part it can be seen that $A=Q\hat{S}\text{diag}(T)\hat{S}^{-1}Q^T$, then the result is proven. If all of $t_{ii}$ are equall, then the equality $T\hat{S}=\hat{S}\text{diag}(T)$ implies that all off-dioagonals of $T\hat{S}=\hat{S}\text{diag}(T)$ must be zero so the result is obtained.
\end{proof}
For a matrix $S\in \text{EM}^n_\alpha$ since $S^{-1}SS^{T}S= G_\alpha$, then the matrix $SS^T$ is similar to $G_\alpha$. Then from Lemma \ref{Leig}, $\sigma (SS^T)=\{1-\alpha ,1+(n-1)\alpha\}$ with the algebraic multiplicity $n-1$ at $1-\alpha$. Moreover from Proposition \ref{Tinv}
\begin{equation}
(SS^T)^{-1}=S^{-T}S^{-1}=\beta(\beta^{-1/2}S^{-T} \beta^{-1/2}S^{-1})=\beta(\tilde{S} \tilde{S}^T), \label{rowe}
\end{equation}
where $\tilde{S}=\beta^{-1/2}S^{-T}\in \text{EM}^n_{\alpha^{\prime}}$. Now suppose that $A\in\mathbb{R}^{n\times n}$ is symmetric matrix with two eigenvalues $\lambda_1,~\lambda_2=\lambda_1 +n(1-\lambda_1)$ with the algebraic multiplicity $n-1$ at $\lambda_1$, where $0<\lambda_1<1<\lambda_2$. Taking $\alpha=1-\lambda_1$, implies that $\sigma (A)=\{1-\alpha ,1+(n-1)\alpha\}$. Therefore for any $S\in EM^n_\alpha$, $A$ is orthogonally similar to $SS^T$. Hence from the Schur form $A=Q(SS^T)Q^T$, where $Q$ is orthogonal matrix. Then $A=(QS)(QS)^T=\tilde{S}\tilde{S}^T$, where $\tilde{S}\in \text{EM}^n_\alpha$. Now we can conclude the following Lemma.
\begin{lemma}
Suppose that a nonsingular symmetric matrix $A\in\mathbb{R}^{n\times n}$ has two distinct eigenvalues $\lambda_1$, $\lambda_2$ with the same signs by the algebraic multiplicity $n-1$ at $\lambda_1$. Then there exists a nonzero $r\in\mathbb{R}$ where $A$ can be factorized as $rSS^T$, uniquely where $S$ is an equiangular matrix.\label{LeiS}
\end{lemma}
\begin{proof}
There are two cases:
\begin{enumerate}
\item $\vert\lambda_1\vert<\vert\lambda_2\vert$\\
Two equations $r(1-\alpha)=\lambda_1$ and $r(1+(n-1)\alpha)=\lambda_2$ have the solution $\alpha=\tfrac{\lambda_2-\lambda_1}{\lambda_2-\lambda_1+n\lambda_1}\in(0,1)$ and $r=\tfrac{\lambda_2-\lambda_1+n\lambda_1}{n}$. Then two eigenvalues $\lambda_1,\lambda_2$ are obtained. As noted before, $A$ can be factorized as $rSS^T$, where $S\in \text{EM}^n_\alpha$.
\item $\vert\lambda_2\vert<\vert\lambda_1\vert$\\
$A^{-1}$ has two eigenvalues $\lambda_1^{-1},\lambda_2^{-1}$. From the previous case $A^{-1}=r^{\prime}\hat{S}\hat{S} ^T$, with the corresponding $r^{\prime},\alpha^\prime$. Then from \eqref{rowe} $A=\dfrac{1}{r^\prime}(\hat{S}\hat{S}^T)^{-1}=\dfrac{\beta}{r^\prime}(SS^T)=rSS^T$, where $S\in \text{EM}^n_{\alpha^\prime}, \alpha^\prime =\frac{-\alpha}{1+(n-2)\alpha}$ and $r=\beta/r^{\prime}$.
\end{enumerate}
\end{proof}
\begin{theorem} 
(Generalization of the symmetric Schur form) Given a symmetric matrix $A\in\mathbb{R}^{n\times n}$ with at most $n-2$ zero eigenvalues and distinct nonzero eigenvalues. Then there are a matrix $S\in \text{EM}^n_\alpha$ with a real $0<\alpha<1$ in a neighborhood of zero and a real diagonal $D=\text{diag}(d_1,\ldots,d_n)$ so that $A=SDS^T$.Also if $\sigma(A)=\{\lambda_1,\ldots ,\lambda_n\}$, then $\sum d_i=\sum \lambda_i$. \label{SDST}
\end{theorem} 
\begin{proof}
without loss of generality we assume thet $A$ is nonsingular with distinct eigenvalues. Because for $A=\text{diag}(\lambda_1,\ldots ,\lambda_k,0,\ldots,0)$ if $\text{diag}(\lambda_1,\ldots ,\lambda_k)=S_kD_kS^T_k$, where $D_k=\text{diag}(d_1,\ldots ,d_k)$, then one can obtain the decomposition $A=SDS^T$ as follows
\begin{equation}
A=\left[\footnotesize{\begin{array}{c:c}
\lambda _1\qquad &~\\\ddots &0\\ \qquad\lambda _k&~\\ \hdashline ~&~\\0& 0 \\ \end{array}}\right]= \left[\footnotesize{\begin{array}{c:c}
~ &~\\S_k&~\\ ~&\tilde{S}_{n-k} \\ -------&~\\0 &~ \end{array}}\right] \left[\footnotesize{\begin{array}{c:c}
d_1\qquad &~\\\ddots &0\\ \qquad d_k&~\\ \hdashline ~&~\\0& 0 \\ \end{array}}\right] \left[\footnotesize{\begin{array}{c:c}
~ &~\\S_k&~\\ ~&\tilde{S}_{n-k} \\ -------&~\\0 &~ \end{array}}\right]^T,
\end{equation}
where $S_k$ is extended by obtaining $\tilde{S}_{n-k}$ from the vectors $e_{k+1},\ldots ,e_n$ by means of the equiangular Algorithm.\\
Now if we find a diagonal matrix $D$ so that the matrix $\bar{S}_\alpha D\bar{S}_\alpha$ is orthogonally similar to $A$ as $A=P\bar{S}_\alpha D\bar{S}_\alpha P$, then the proof is completed: since $P\bar{S}_\alpha\in EM_\alpha^n$, then we set $S=P\bar{S}_\alpha$ so $A=SDS^T$.
\\For obtaining $D$ we can write $\bar{S}_\alpha D\bar{S}_\alpha=\bar{S}_\alpha D{\bar{S}_\alpha}^2{\bar{S}_\alpha} ^{-1}=\bar{S}_\alpha DG_{\alpha}{\bar{S}_\alpha}^{-1}$, which indicates that $DG_{\alpha}$ must be similar to $\Lambda$. So the characteristic polynomial of them are the same. We discuss the characteristic polynomial of $DG_{\alpha}$ as follows
\begin{equation}
\text{det}(xI-DG_{\alpha})=\text{det} \footnotesize{\begin{bmatrix}x-d_{1}&-\alpha d_{1}&\cdots &-\alpha d_{1}\\-\alpha d_{2} &x-d_{2}&~&\vdots\\\vdots &~&\ddots &~\\-\alpha d_{n} &\cdots &~&x-d_{n}\end{bmatrix}} =(-1)^{n}\alpha ^{n}(\prod_{i=1}^{n}d_{i})~\text{det} \footnotesize{ \begin{bmatrix}\tfrac{x-d_{1}}{-\alpha d_{1}}&1&\cdots &1\\1&\tfrac{x-d_{2}}{-\alpha d_{2}}&~&\vdots\\\vdots &~&\ddots &1\\1&\cdots &1&\tfrac{x-d_{n}}{-\alpha d_{n}} \end{bmatrix}}\cdot \label{detGn1}
\end{equation}
Let $G_n=\footnotesize{\begin{bmatrix} x_1&1&\cdots&1\\1&x_2&~&\vdots\\\vdots &~&\ddots&1 \\1&\cdots&1&x_n \end{bmatrix}}$ with $x_i=\dfrac{x-d_{i}}{-\alpha d_{i}}$. Then the following recurrence is obtained\\ $\text{det}(G_n)=(x_n-1)~\text{det}(G_{n-1})+(-1)^{n-1}(1-x_1)~\text{det}(\tilde{G}_{n-1})$, where $\tilde{G}_{n-1}=\footnotesize{
\begin{bmatrix}1&1&\cdots&1\\x_2&1& ~&1\\\vdots &~&\ddots &\vdots\\1&\cdots &x_{n-1}&1\end{bmatrix}}$. The term $\text{det}(\tilde{G}_{n-1})$ can be computed recursively as
\begin{equation}
\text{det}(\tilde{G}_{n-1})=-(x_{n-1}-1)~\text{det} (\tilde{G}_{n-2})=\cdots =(-1)^{n-2}(x_2-1)\cdots(x_{n-1}-1)=(-1)^{n-2} \prod _{i=2}^{n-1} (x_i-1)\cdot
\end{equation}
Since $\text{det}(G_2)=x_1x_2-1=(x_1-1)(x_2-1)+(x_1-1) +(x_2-1)$, then we can write
\begin{align}
\text{det}(G_n)&=(x_n-1)~\text{det}(G_{n-1})+(-1)^{n-1}(1-x_1)\cdot (-1)^{n-2}\prod_{i=2}^{n-1} (x_i-1)\nonumber\\&=(x_n-1)\cdot \text{det}(G_{n-1})+\prod_{i=1}^{n-1} (x_i-1)=\prod_{i\neq n} (x_i-1)+\cdots +\prod_{i\neq 1} (x_i-1)+\prod_{i=1}^{n} (x_i-1)\cdot \label{detGn2}
\end{align}
Eventually replacing \eqref{detGn2} in \eqref{detGn1} and simplifying results the following equations.
\begin{align}
\text{det}(xI-DG_{\alpha})=&(-1)^{n}\alpha ^{n}(\prod_{i=1}^{n}d_{i})\left[\prod_{i\neq n}(\tfrac{x-d_{i}}{-\alpha d_{i}}-1)+\cdots +\prod_{i\neq 1}(\tfrac{x-d_{i}}{-\alpha d_{i}}-1)+\prod_{i=1}^n (\tfrac{x-d_{i}}{-\alpha d_{i}}-1)\right]\nonumber\\=& (-1)^{n}\alpha ^{n}(\prod_{i=1}^{n}d_{i})~\frac{1}{(-1)^{n}\alpha ^{n}(\prod_{i=1}^{n}d_{i})}~[(-\alpha d_{n}\prod_{i\neq n}(x-d_{i}(1-\alpha))\nonumber\\&- \cdots-\alpha d_{1}\prod_{i\neq 1}(x-d_{i}(1-\alpha))+\prod_{i=1}^{n}(x-d_{i}(1-\alpha))]=\cdots
\end{align}
\begin{align}
=&x^{n}+(-\alpha +\alpha -1)(\sum_{i=1}^{n}d_{i})x^{n-1}+(-2\alpha (\alpha -1)+(\alpha -1)^2)(\sum_{1\leq i< j\leq n}d_{i}d_{j})x^{n-2}\nonumber\\&+\cdots +(-(n-1)\alpha (\alpha -1)^{n-2}+(\alpha -1)^{n-1})(\sum _{1\leq i_{j}\leq n}d_{i_{1}}\cdots d_{i_{n-1}})x+(-n\alpha (\alpha -1)^{n-1}+(\alpha -1)^n)d_{1}\cdots d_{n}\nonumber\\=&x^{n}- (\sum_{i=1}^{n} d_{i})x^{n-1}+(1-\alpha ^2)(\sum_{1\leq i< j\leq n}d_{i}d_{j}) x^{n-2}-\cdots -(\alpha -1)^{n-2} ((n-2)\alpha +1)(\sum_{1\leq i_{j}\leq n}d_{i_{1}}\cdots d_{i_{n-1}}) x\nonumber\\&-(\alpha -1)^{n-1}(1+(n-1)\alpha) d_{1}\cdots d_{n}\cdot
\end{align}
On the other hand, the characteristic polynomial of $\Lambda$ is illustrated as follows
\begin{align}
p(x)=&(x-\lambda_{1})(x-\lambda _{2})\cdots(x-\lambda_{n})=x^{n}-(\sum _{i=1}^{n}\lambda_{i})x^{n-1}+ (\sum_{1\leq i< j\leq n}\lambda_{i}\lambda_{j})x^{n-2} \nonumber\\&-\cdots +(-1)^{n-1}(\sum_{1\leq i_{j}\leq n}\lambda_{i_{1}}\cdots \lambda_{i_{n-1}})x +(-1)^{n}\lambda _{1}\cdots \lambda_{n}\cdot
\end{align}
Since det$(xI-DG_\alpha)=\text{det}(xI-\Lambda)$, then
\begin{align}
&\sum_{i=1}^{n}d_{i}=\sum_{i=1}^{n} \lambda_{i}=c_1 \Rightarrow \text{trace}(D)=\text{trace}(\Lambda),\nonumber \\& \sum_{1\leq i< j\leq n}d_{i}d_{j}=\frac{1}{1-\alpha ^2}\sum_{1\leq i< j\leq n}\lambda_{i}\lambda_{j}=c_2,\nonumber\\&\qquad \vdots\nonumber\\& \sum_{1\leq i_{j}\leq n}d_{i_{1}}\cdots d_{i_{n-1}}=\frac{1}{(1-\alpha)^{n-2}(1+(n-2)\alpha)}\sum_{1\leq i_{j}\leq n}\lambda_{i_{1}}\cdots \lambda_{i_{n-1}}=c_{n-1},\nonumber\\&d_{1}\cdots d_{n}=\frac{1}{(1-\alpha)^{n-1}(1+(n-1)\alpha)}\cdot\lambda_{1}\cdots \lambda_{n}=c_{n}\cdot \label{d_i}
\end{align}
So $d_{1},~d_{2},\ldots,d_{n}$ are the roots of the following polynomial 
\begin{equation}
g(x)=x^{n}-c_{1}x^{n-1}+\cdots+(-1)^{n-1}c_{n-1} x+ (-1)^{n}c_{n}\cdot \label{g_n}
\end{equation}
Then the necessary condition for implementation of this factorization, is that the roots of $g(x)$ are all real, because $SDS^T$ must be symmetric. The scalar $\alpha$ in the coefficients of $g(x)$ can be considered as the perturbation in those of $p(x)$. Since the roots of $p(x)$ are distinct, then by the ``intermediate value theorem" there is a $\alpha$ in the neighborhood of zero for which the roots of $g(x)$ are all ``real" and possibly distinct. One can use an argument from the complex analysis \cite{Ahlfors}: the eigenvalues of $A$ are the continuous function of $A$, even though they are not differentiable.
\end{proof}
One can provide a counter example for while the eigenvalues of $A$ are not distinct: let $A=\Lambda=\text{diag}(0,1,1)$ then $D=\text{diag}(0,1+\sqrt{ \tfrac{\alpha ^2}{1-\alpha ^2}}i,1- \sqrt{\tfrac{\alpha ^2}{1-\alpha ^2}}i)$ which is not real and symmetric. So $A$ wont be factorized as $SDS^T$. In general if $A$ be a factor of identity matrix: let $A=rI=SDS^T$, then $rS^{-1}S^{-T}=rG_\alpha ^{-1}=D$ that satisfies $G_\alpha$ is diagonal which is a contradiction. Therefore in this case the corresponding $g(x)$ in \eqref{g_n} has at least two nonreal roots. Note that distinction of the eigenvalues of $A$ is the sufficient condition for the Theorem \ref{SDST} but not the necessary. For example if $A$ be symmetric by two nonzero eigenvalues with the algebraic multiplicity $n-1$ at one of them, then by the Lemma \ref{LeiS} the decomposition $A=rSS^T$ is possible and it suffics to let $D=rI$.
\begin{proposition}
If $r\neq 0$ and $0<\alpha<1$, then the following polynomial has at least two nonreal roots. $(n\geq 2)$\label{root}
\begin{equation}
g_{n}(x)=x^n-nrx^{n-1}+\dfrac{\binom{n}{2}r^2}{1-\alpha ^2}x^{n-2}-\dfrac{\binom{n}{3}r^3}{(1-\alpha)^2(1+2\alpha)}x^{n-3} +\cdots +(-1)^n\dfrac{r^n}{(1-\alpha)^ {n-1}(1+(n-1)\alpha)}\cdot
\end{equation} \label{Pro}
\end{proposition}
\begin{proof}
Since the Theorem \ref{SDST} does not hold for $A=rI$, then there is no real symmetric $D$ satisfies $A=SDS^T$. So using \eqref{d_i} the coefficients of $g(x)$ are computed as follows
\begin{equation}
c_1=\sum _{i=1}^n r =nr,\quad c_2=\frac{1}{1-\alpha ^2}\sum _{1\leq i<j\leq n}r^2 =\dfrac{\binom{n}{2}r^2}{1-\alpha ^2}~,\quad\cdots\quad ,\quad c_n =\frac{r^n}{(1-\alpha)^{n-1}(1+(n-1)\alpha)}\cdot
\end{equation}
\end{proof}
\begin{example}
One can check the accuracy of Lemma \eqref{root} for the cases $n=2,3$ with a nonzero real $r$: for $n=2$ that $g_2 (x)=x^2 -2rx+\frac{r^2}{1-\alpha ^2}$, the two roots of $g_2$ are nonreal: $x_{1,2}=r\pm \frac{r\alpha}{\sqrt{1-\alpha ^2}}i$ and $g_2>0$. For $n=3$, $g_3(x)=x^3 -3rx^2 +\frac{3r^2}{1-\alpha^2}x-\frac{r^3}{(1-\alpha)^2 (1+2\alpha)}$ then $g_3^\prime(x)=3x^2-6rx +\frac{3r^2}{1-\alpha^2}$. Clearly $g_3^\prime=3g_2>0$. Therefore $g_3$ is an increasing function with only one real root. In general, $g_n^\prime=ng_{n-1}$. By induction $g_{n-1}$ has at most $n-3$ real roots and by intermediate value theorem $g_n$ has at most $n-2$ real roots. \qquad\qquad\qquad\qquad\qquad\qquad \qquad $\lozenge$
\end{example}
\begin{example}
Assume that $\Lambda =\text{diag}(1,2,3)$. Using Theorem \ref{SDST} one can find a bound for $\alpha$ for which the factorization $\Lambda=SDS^T$ holds for some $S\in \text{EM}^n_\alpha$. So we can write $p(x)=x^3 -6x^2 +11x-6$ and $g(x)=x^3 -6x^2 +\frac{11}{1-\alpha^2}x-\frac{6}{(1-\alpha)^2(1+2\alpha)}$. Considering MATLAB $\mathtt{roots}$ function, the roots of $g$ are all real if $\alpha\leq 0.1843$, by rounding.\qquad\qquad\qquad\qquad\qquad \qquad\qquad\qquad\qquad \qquad\qquad\qquad\qquad\qquad \qquad\qquad\qquad\qquad \qquad\qquad\qquad\qquad\quad~~ $\lozenge$
\end{example}
As noted before Lemma \ref{LeiS} shows that Theorem \ref{SDST} holds for a class of symmetric matrices with two nonzero eigenvalues with the multiplicity $n-1$ for one of them. In these cases the corresponding diagonal matrix $D$ is a factor of identity matrix. 
\begin{example}
If $\Lambda =\text{diag}(1,1,2)$, then from Lemma \ref{LeiS} $r=4/3$ and $\alpha =1/4$ that $D$ is the same $(4/3)I_3$. Also by \eqref{s,t} $s=0.9856$ and $t=0.1196$ by rounding. Therefore $\bar{S}_\alpha=0.9856G_{0.1213}$ and Theorem \eqref{SDST} results\\
$P=\left[\footnotesize{\begin{array}{ccc} ~~0.8059&-0.1310&0.5774\\-0.2895&~~ 0.7634& 0.5774\\-0.5164&-0.6325&0.5774 \end{array} }\right]$ and $S=P^T \bar{S}_\alpha=\left[\footnotesize{\begin{array}{ccc}~~0.6979&-0.2507&-0.4472\\-0.1134& ~~0.6612&-0.5477\\~~0.7071&~~0.7071& ~~0.7071\end{array} }\right]$. Then $\Lambda=SDS^T$. \qquad ~~$\lozenge$
\end{example}
We provide two examples of special cases of Lemma \ref{Pro}.
\begin{example}
Let $r=1-\alpha$ and by assumption that $0<d=\frac{\alpha}{1-\alpha}<\infty$, the general term of $g_n$ is $c_k =\frac{\binom{n}{k}r^k}{(1-\alpha)^{k-1}(1+(k-1)\alpha)}= \frac{\binom{n}{k}(1-\alpha)}{1+(k-1)\alpha}\frac{\binom{n}{k}}{1+dk}$. Therefore $g_n(x)= x^n-\sum_{k=1}^n \frac{\binom{n}{k}}{a_k}x^{n-k}$ has at least two nonreal roots, where $a_k$ is an arithmetic sequence with the initial term $a_1=1+d>1$ and the common difference $d$.
\end{example}
\begin{example}
Let $r=1+(k-1)\alpha$ and by assumption that $0<d=\frac{\alpha}{1-\alpha}<\infty$, the general term of $g_n$ is $c_k=\frac{\binom{n}{k}r^k}{(1-\alpha)^{k-1}(1+(k-1)\alpha)}= \binom{n}{k}(\frac{1+(k-1)\alpha}{1-\alpha})^{k-1}=\binom{n}{k}(1+dk)^{k-1}$. Therefore $g_n(x)= x^n-\sum_{k=1}^n \binom{n}{k} (a_k)^{n-1}x^{n-k}$ has at least two nonreal roots, where $a_k$ is an arithmetic sequence with the initial term $a_1=1+d>1$ and the common difference $d$.
\end{example}
Another strong counter example is the case that $A$ has an eigenvalue with multiplicity less than $n-1$, where $n$ is the number of nonzero eigenvalues. Without loss of generality assume that $A=\text{diag}(\lambda_1,\lambda_2,\ldots,\lambda_n)$ is nonsingular where $\lambda_1=\cdots=\lambda_i$ $(i<n-1)$ and $n>3$. We can prove it in the following lemma.
\begin{lemma}
If the nonzero eigenvalues of $A$ is $\lambda_1,\lambda_2,\ldots,\lambda_n$ $(n>3)$ where the $n-k$ of $\lambda_i$'s are equal as $\lambda_{i_1}=\cdots=\lambda_{i_{n-k}}$ $(2\leq k\leq n-2)$. Then there are no $S\in \text{EM}^n_\alpha$ with $0<\alpha<1$ and real diagonal $D$ satisfy $A=SDS^T$. \label{naghz}
\end{lemma}
\begin{proof}
Without loss of generality assume that $A=\text{diag}(\lambda_1,\ldots,\lambda_n)$ is diagonal with no zero eigenvalue. For simplicity we suppose that $\lambda_{k+1}=\cdots=\lambda_n =\lambda$. Otherwise by multiplying $A$ by permutation matrices from left and right the desired form is obtained. Let us suppose for a contradiction that there are $S$ and $D$ satisfy the hypothesize of the problem so $S^{-1}AS^{-T}=D$. From the equations \eqref{ninv} and \eqref{sipsjp} and its outcome in section \ref{inveig}, $S^{-1}$ has the equiangular rows with the norm $\sqrt{\beta}$ and the cosine of the angle $\alpha^\prime<0$. Then analogous to \eqref{rowe} there exists $\tilde{S}\in \text{EM}^n_{\alpha^\prime}$ so that $S^{-1}=\sqrt{\beta}\tilde{S}^T$. Therefore $\tilde{S}^T\text{diag}(\lambda_1,\ldots,\lambda_k,\lambda,\cdots,\lambda)\tilde{S}=\frac{1}{\beta}D$. The subtraction of this equality from the equation $\lambda\tilde{S}^T I_n\tilde{S}=\lambda G_{\alpha^\prime}$ is illustrated as follows
\begin{equation}
\begin{bmatrix}(\lambda _1-\lambda)\tilde{s}_{11}&\ldots &(\lambda _{k}-\lambda)\tilde{s}_{k1}\\ \vdots &\ddots &\vdots\\ (\lambda _1-\lambda)\tilde{s}_{1n} &\ldots &(\lambda _{k}-\lambda)\tilde{s}_{kn} \end{bmatrix}\cdot \begin{bmatrix}\tilde{s}_{11}&\ldots &\tilde{s}_{1n}\\ \vdots &\ddots &\vdots\\ \tilde{s}_{k1} &\ldots &\tilde{s}_{kn} \end{bmatrix}= \footnotesize{\begin{bmatrix}\tfrac{d_1}{\beta}-\lambda &-\lambda\alpha ^\prime &\cdots &-\lambda\alpha ^\prime\\-\lambda\alpha ^\prime &\tfrac{d_2}{\beta}-\lambda &~&\vdots\\\vdots &~&\ddots &~\\-\lambda\alpha ^\prime &\cdots &~ &\tfrac{d_n}{\beta}-\lambda  \end{bmatrix}}\cdot\label{STAS}
\end{equation}
The right hand matrix is analogous with $G_n$ in \eqref{detGn1} and its rank must be at most $k$. On the other hand the assumed matrix decomposition is possible for the numbers between zero and $\alpha$. So if $\alpha$ converges to zero, then the diagonal entries of the above matrix tend to $d_i-\lambda$ and all off-diagonals tend to zero. One can consider a subsequence of the $\alpha$'s close to zero in such a way that for each $i$, $\frac{d_i}{\beta}-\lambda\neq 0$ and also $\vert (n-1)\alpha^\prime\lambda\vert<\vert \frac{d_i}{\beta}-\lambda\vert$. Then there exists a $0<\alpha<1$ for which 
\begin{equation}
0<\vert (n-1)\alpha^\prime\lambda\vert<\vert \frac{d_i}{\beta}-\lambda\vert\cdot
\end{equation}
Now by the Gerschgorin Theorem \cite{Horn,Meyer} the eigenvalues of the mentioned matrix are nonzero. So its rank will be $n>k$ which is a contradiction.
\end{proof}
Now the outcome of mentioned counter example can be represented as a theorem in related to the general form of the polynomials with nonreal roots.
\begin{theorem}
The real scalers $\lambda_1,\lambda_2,\ldots,\lambda_n$ and $0<\alpha<1$ $(n>3)$ are given so that there are two cases for $\lambda_i$'s: either all them are equal or the $i$ of them are equal $(2\leq i\leq n-2)$. Then the following polynomial has at least two nonreal roots. $(n\geq 2)$
\begin{equation}
f_{n}(x)=x^n-(\sum _{i=1}^n\lambda _i)x^{n-1}+\dfrac{\sum _{1\leq i<j\leq n}\lambda _i\lambda _j}{(1-\alpha)(1+\alpha)}x^{n-2}+\cdots +(-1)^n\dfrac{\lambda _1\ldots\lambda _n}{(1-\alpha)^ {n-1}(1+(n-1)\alpha)}\cdot \label{fnx}
\end{equation}
Moreover, all real polynomials of degree $n$ which has the nonreal roots, can be illustrated as the form of \eqref{fnx}. 
\end{theorem}
\begin{proof}
From the Lemmas \ref{root} and \ref{naghz} the first part is proven. For the next part as regarding to the Lemma \ref{LeiS} and the Theorem \ref{SDST} when the $\lambda_i$'s are distinct or $n-1$ of them are equal, then for a $\alpha$ all $d_i$'s are real. So we can say that if the scalers $d_i$ with at least two of them are complex which are the roots of $f_n(x)$, then the $\lambda_i$'s must be as mentioned in the hypothesize of the theorem.
\end{proof}
\section{Doubly equiangular matrices}
In this section we study the special type of equiangular matrices which have equiangular rows, in addition to the equiangular columns. The matrix $\bar{S}_\alpha$ in section \ref{inveig} is of this type. But we want to find the general form of them.
\begin{definition}
The matrix $\bar{S}$ is called ``doubly equiangular" if for $0\neq \alpha\in(\frac{-1}{n-1},1)$, $\bar{S},\bar{S}^T \in \text{EM}^n_\alpha$. Also $\text{DEM}^n_\alpha$ denotes the set of all doubly equiangular matrices with the cosine of the angle between its row or column vectors $\alpha$. \label{ddoub}
\end{definition} 
Note that a doubly equiangular matrix $\bar{S}$ is normal because $\bar{S}\bar{S}^T=\bar{S}^T\bar{S}$. So $\bar{S}$ is orthogonally diagonalizable. From Schur form $\bar{S}=Q\Lambda Q^T$ that $\Lambda$ is a blocked diagonal matrix with $1\times 1$ and $2\times 2$ blocks. We write $\bar{S}^T\bar{S}=G_\alpha =Q\Lambda^T\Lambda Q^T$, where $\Lambda_\alpha=\Lambda^T\Lambda$ is a diagonal matrix with some eigenvalues corresponding to the $1\times 1$ blocks of $\Lambda$. As regards to the Lemma \ref{Leig}, $\Lambda_\alpha$ has an eigenvalue $1+(n-1)\alpha$ in a $1\times 1$ block with the corresponding eigenvector $e$. So we can say its corresponding eigenpair in $\bar{S}$ is $(\sqrt{1+(n-1)\alpha},e)$. Since $e$ is an eigenvector of $\bar{S}$, then the row sum of $\bar{S}$ is the same. The same reasoning is true for $\bar{S}^T$. So $e$ is also an eigenvector of $\bar{S}^T$ which implies that the column sum of $\bar{S}$ is the same. Then $\bar{S}e=\bar{S}^Te=(1+(n-1)\alpha)^ {1/2}e$. For the case of $\alpha=0$ we present a special definition similar to the before.
\begin{definition}
The orthogonal matrix $\bar{Q}$ is called ``doubly orthogonal" if $e$ be the eigenvector of $\bar{Q}$. Also $\text{DOM}^n$ denotes the set of all doubly orthogonal matrice.
\end{definition}
In this case the condition of being normal a matrix $Q$ is not sufficient for $e$ to be its eigenvector. Because $G_\alpha=I$ so any vector is in the eigenspace of eigenvalue $1$ so that the eigenvectors of $I$ and accordingly $Q$ are not restricted to $e$. Although for any orthogonal matrix $Q$, $Q^TQ=QQ^T$, but it may not be doubly orthogonal necessarily based on definition, unless the vector $e$ be its eigenvector.\\
Now we want to obtain a doubly equiangular (doubly orthogonal) matrix like $\bar{S}$ $(\bar{Q})$ by means of an equiangular (orthogonal) matrix $S$ $(Q)$ so that its column sum vector be in direction to $e$. $(\bar{S}e=\lambda e$ or $\bar{Q}e=\lambda e)$.
\begin{theorem}
If $S\in \text{EM}^n_\alpha$ with $\alpha\in(\frac{-1}{n-1},1)$. Then one can obtain a doubly equiangular (orthogonal) matrix like $\bar{S}\in \text{DEM}^n_\alpha ~(\text{DOM}^n)$ from $S$ as
\begin{equation}
\bar{S}=(I-\frac{2uu^T}{\Vert u\Vert ^2})S\label{dequ}
\end{equation}
where $u=(S-\sqrt{1+(n-1)\alpha}I)e$. \label{dthe}
\end{theorem}
\begin{proof}
It suffics to construct a Householder transformation like $H$ so that transform the column sum vector of $S$ to a vector in direction to $e$. The subtraction of normalized column sum vectors of $S$ and $\bar{S}$ is $\frac{1}{\sqrt{n(1+(n-1)\alpha)}}Se-\frac{1}{\sqrt{n}}e$ which must be in direction to $u$. Then $\bar{S}$ is obtained by multiplication of the Householder matrix $I-\frac{2uu^T}{\Vert u\Vert ^2}$ to $S$ from the left. It can be checked that $\bar{S}\bar{S}^T=G_\alpha$ and $\bar{S}e=\bar{S}^Te=(1+(n-1)\alpha)^ {1/2}e$.
\end{proof}
As noted in the Example \ref{psquare}, the sufficient condition for producing a doubly equiangular matrix with nonnegative entries is that $\frac{n-2}{n-1}\leq\alpha<1$ and the lower bound of $\alpha$ is hold for $\sqrt{G_\alpha}$ with $\tfrac{1}{\sqrt{n}}<t\leq\tfrac{1}{\sqrt{n-1}}$. Actually $\sqrt{G_\alpha}$ in this case is a symmetric matrix, but not the principal root of $G_\alpha$. The general form of $\sqrt{G_\alpha}$ is illustrated as follows
\begin{equation}
\sqrt{G_\alpha}=\left[\begin{array}{cccc}0&\frac{1}{\sqrt{n-1}}&\ldots &\frac{1}{\sqrt{n-1}}\\\frac{1}{\sqrt{n-1}}&0&\ldots &\frac{1}{\sqrt{n-1}}\\\vdots &\vdots &\ddots &\vdots\\\frac{1}{\sqrt{n-1}}&\frac{1}{\sqrt{n-1}}&\ldots &0\end{array}\right]~,\qquad\alpha =\frac{n-2}{n-1}\cdot
\end{equation}
As regard to the matrix $\bar{S}_\alpha=G_\alpha^{1/2}$ which has positive entries, we deduce that the mentioned condition is not necessary. Actually matrix $(1+(n-1)\alpha)^ {-1/2}\bar{S}$ where $\bar{S}\in \text{DEM}^n_\alpha$ with the mentioned condition is a doubly stochastic matrix and without it, is a quasi doubly stochastic matrix.\\
One can provide an algorithm named DEA followed by the Theorem \ref{dthe} to produce a matrix $S\in \text{DEM}^n_\alpha~(Q\in \text{DOM}^n)$ from the decomposition SR (QR) of a nonsingular matrix $A$.
\begin{algorithm} 
\caption{.~This algorithm produces a doubly equiangular (orthogonal) matrix $\bar{S}~(\bar{Q})$ with $\alpha\in(\frac{-1}{n-1},1)$ from a nonsingular matrix $A$ of size $n$.}
\label{algo:dequ}
\begin{algorithmic}[1]
\STATE $A=SR~;$~\% SR (QR) decomposition with $\alpha\in(\frac{-1}{n-1},1)$
\STATE $u=(S-\sqrt{1+(n-1)\alpha}~I)e~;$
\STATE $\bar{S}=(I-\frac{2uu^T}{\Vert u\Vert ^2})S~;$
\end{algorithmic}
\end{algorithm}
\begin{example}
If $A=\left[\footnotesize{\begin{array}{cccc}1&1/2 & 1/3 &1/4\\1/2 & 1/3 &1/4&1/5\\1/3 &1/4&1/5&1/6\\1/4&1/5&1/6&1/7 \end{array}}\right]$, then by applying DEA on $A$, the rounded $\bar{S}$ with $\alpha=2/3$ is obtained as follows
\begin{equation}
\bar{S}=\left[\footnotesize{\begin{array}{cccc}0.8517&0.3048& 0.3774&0.1981\\ 0.3976&0.3942&0.1205&0.8198\\ 0.2399&0.1863&0.8189&0.4869\\0.2429 &0.8468&0.4152&0.2273 \end{array}}\right]~,
\end{equation}
where $\bar{S}\bar{S}^T=\bar{S}^T\bar{S}= G_{2/3}$ and $\bar{S}e=\bar{S}^T e=\sqrt{3}e$.
\end{example}
\begin{example}
The orthogonal matrix $Q=\left[\footnotesize{\begin{array}{ccc}~~3/7&-2/7&~~6/7\\~~6/7&~~3/7 &-2/7\\-2/7&~~6/7&~~3/7\end{array}} \right]$ in Example \ref{Dorth} is a cyclic doubly orthogonal matrix so that $Qe=Q^Te=e$.
\end{example}
\begin{example}
The matrix $Q=\left[\footnotesize{\begin{array}{cccc}-1/2&~1/2&~1/2&~1/2\\-1/2&~1/2 &-1/2&-1/2\\-1/2&-1/2&~1/2 &-1/2\\-1/2&-1/2&-1/2&~1/2 \end{array}} \right]$ is an orthogonal matrix. From DEA, $Q$ is transformed to the following doubly orthogonal matrix
\begin{equation}
\bar{Q}=\left[\footnotesize{\begin{array}{cccc}-1/2&1/2&~1/2&~1/2\\~1/2&~5/6& -1/6&-1/6\\~1/2&-1/6&~5/6 &-1/6\\~1/2&-1/6&-1/6&~5/6 \end{array}} \right],
\end{equation}
so that $\bar{Q}e=\bar{Q}^Te=e$.
\end{example}
From the later discussions and examples one can conclude that doubly orthogonal matrices are quasi doubly stochastic matrices.
\begin{lemma}
The matrix $S_\alpha\in \text{EM}^n_\alpha$ is commutable with family of all $\bar{S}\in \text{DEM}^n_{\alpha^\prime}$, where $\alpha,\alpha^\prime\in(\frac{-1}{n-1},1)$. The matrix $G_\alpha$ has also the same property. \label{comm}
\end{lemma}
\begin{proof}
It suffics to let $S_\alpha=(s-t)I+tee^T$, where $s,t$ are defined in \eqref{s,t2}. So we can write
\begin{equation}
S_\alpha\bar{S}=((s-t)I+tee^T)\bar{S}=(s-t)\bar{S}+t(1+(n-1)\alpha)^{1/2}ee^T= \bar{S}S_\alpha.
\end{equation}
Also $G_\alpha\bar{S}=(S_\alpha)^2\bar{S}=S _\alpha\bar{S}S_\alpha= \bar{S}(S_\alpha)^2=\bar{S}G_\alpha$.
\end{proof}
\begin{proposition}
Let $\bar{S}_1\in \text{DEM}^n_{\alpha_1}$ and $\bar{S}_2\in \text{DEM}^n_{\alpha_2}$ with $\alpha_1,\alpha_2\in(\frac{-1}{n-1},1)$. Then $\bar{S}_1\bar{S}_2$ is a factor of a doubly equiangular matrix and so is normal.
\end{proposition}
\begin{proof}
From Lemma \ref{comm} we can write
\begin{equation}
\bar{S}_1\bar{S}_2(\bar{S}_1\bar{S}_2)^T =\bar{S}_1\bar{S}_2 \bar{S}_2^T\bar{S}_1^T=\bar{S}_1 G_{\alpha_2}\bar{S}_1^T=G_{\alpha_2} \bar{S}_1\bar{S}_1^T= G_{\alpha_2}G_{\alpha_1}=cG_{\alpha ^\prime},
\end{equation}
where $c=1+(n-1)\alpha_1 \alpha _2$ and $\alpha^\prime=(\alpha_1 +\alpha_2 +(n-2)\alpha_1 \alpha_2)/c$. Proving $(\bar{S}_1\bar{S}_2)^T\bar{S}_1\bar{S} _2=cG_{\alpha^\prime}$ is similar the last one.
\end{proof}
In general, DEMs (DOMs) are not commutable. Actually being normal the multiplication of two matrices of DEMs (DOMs), is necessary condition for commuting but not sufficient.\\
In the category of normal matrices, there are symmetric (hermitian), skew-symmetric (skew-hermitian) and orthogonal (unitary) matrices. Now we can add the new type of normal matrices ``doubly equiangular (orthogonal) matrices" to that category. However there exist normal matrices that are not included any of these categories.
\section{Equiangular vectors as a sequence of Equiangular frames} \label{equiframe}
In this section another aspect of equiangular vectors is considered, which is the possibility of assuming them as a equiangular frame (EF). The theory of frames plays a fundamental role in the signal processing, image processing, data compression and more which defined by Duffin and Schaeffer \cite{Duffin}. 
\begin{definition}
A sequence $f_n~(n\in\mathbb{Z})$ of elements in a Hilbert space $H$ is called a frame if there are constants $c_1,c_2>0$ so that
\begin{equation}
c_1\Vert f\Vert ^2\leq \sum _{n\in\mathbb Z}\vert\langle f,f_n \rangle\vert^ 2\leq c_2\Vert f\Vert ^2,\quad for~all~f\in H.
\end{equation} \label{Deff}
\end{definition}
The scalars $c_1$ and $c_2$ are called lower and upper frame bounds, respectively. The largest $c_1>0$ and the smallest $c_2>0$ satisfying the frame inequalities for all $f\in H$ are called the optimal frame bounds. The frame is a tight frame if $c_1=c_2$ and a normalized tight frame or Parseval frame if $c_1=c_2=1$. A frame is called overcomplete in the sense that at least one vector can be removed from the frame and the remaining set of vectors will still form a frame for $H$ (but perhaps with different frame bounds). Equiangular tight frames (ETFs) potentially have many more practical and theoretical applications \cite{Bodmann,Holmes}. An equiangular tight frame is a set of vectors $\{f_i\}_{i=1}^m$ in $\mathbb{R}^n$ (or $\mathbb{C}^n$) that satisfies the following conditions \cite{Strohmer}.
\begin{flushleft}
1.$~\Vert f_i\Vert _2=1 ,\qquad\qquad\qquad i=1,\ldots ,m.$\\
2.$~\vert\langle f_i,f_j\rangle\vert =\alpha ,\qquad\qquad ~~\forall ~i\neq j~\text{and a constant}~\alpha .$\\
3.$(n/m)\sum _{i=1}^m \langle f,f_i\rangle f_i=f,\quad ~~\forall ~\text{f}\in \mathbb{R}^n (\text{or}~\mathbb{C}^n).$
\end{flushleft}
Taking inner product of the equality in third condition with $f$ indicates that the mentioned vectors form a tight frame. We set $S_n=[f_1,\ldots ,f_m]$. It can be checked that $(n/m)S_n {S_n}^T f=f$ so $S_n{S_n}^T =(m/n)I_n$ which is equivalent to the condition 3:
\begin{flushleft}
$3^{'}$.~$S_n{S_n^T} =(m/n)I_n$,\qquad where $S_n=[f_1,\ldots ,f_m]$.
\end{flushleft}
In the next lemma we show that matrix $S_n$ with $m=n+1$ can be illustrated so that has a column vector $e_1$. Then by deletion this vector a $n\times n$ equiangular matrix $S$ is obtained so that $SS^T= \text{diag}(\frac{1}{n},\frac{n+1}{n},\ldots,\frac{n+1}{n})$. As discussed in \eqref{rowe} and the Lemma \ref{LeiS}, $SS^T$ forms a diagonal matrix with eigenvalues $\{1-\alpha,1+(n-1)\alpha\}$ where $\alpha=-\frac{1}{n}$.

The conditions $2,3$ together imply
\begin{equation}
\alpha =\sqrt{\frac{m-n}{n(m-1)}},
\end{equation}
which is the smallest possible $\alpha$ for a set of equiangular normalized vectors in $\mathbb{R}^n$ (or $\mathbb{C}^n$). Due to the theoretical and numerous practical applications, equiangular tight frames are noticeably the most important class of finite-dimensional frames which have many applications in the signal processing, communications, coding theory, sparse approximation and more. \cite{Health,Strohmer}
\begin{example}
(Orthonormal Bases). When $m=n$, the ETFs form orthogonal matrices.
\end{example}
A set of equiangular lines with the angle $0<\theta<\pi/2$ in $\mathbb{R}^n$ is the set of lines which pass through the origin so that the cosine of the angle between them are $\pm\cos\theta =\pm\alpha$. Gerzon \cite{Lemmens} (Theorem 3.5) proved that the number of equiangular lines in $\mathbb{R}^n$ cannot be more than $\binom{n+1}{2}$. Similar proof implies that this maximum number of equiangular lines cannot be more than $n^2$ in the $n$-dimensional complex space $\mathbb{C}^n$. The following lemma from \cite{Casazza} specifies the maximum number of EVs in $\mathbb{R}^n$. 
\begin{lemma}
The maximum number of equiangular lines in $\mathbb{R}^n$ $n\geq 2$ is $n+1$, with the cosine of the angle between them is $-\frac{1}{n}$. \label{Lmax}
\end{lemma}
\begin{proof}
The proof is carried out by induction. If $n=2$, then the desired equiangular lines are in direction to three vectors with the angle $-\pi/3$. It can represented as the normalized column vectors of a $2\times 3$ matrix as 
\begin{equation}
S_2=\left[\begin{array}{ccc}1&-\frac{1}{2}&-\frac{1}{2}\\0&~~\frac{\sqrt{3}}{2}&-\frac{\sqrt{3}}{2}\end{array}\right];\quad n=2~,~\alpha =-\frac{1}{2}~\cdot
\end{equation}
For $n=3$, the desired equiangular lines are in direction to four vectors in such a way that they make a regular triangular pyramid or tetrahedron when they pass through the origin. So the cosine of the angle between them is $-1/3$. It can be represented as the normalized column vectors of a $3\times 4$ matrix as
\begin{equation}
S_3=\left[\begin{array}{cccc}1&-\frac{1}{3}&-\frac{1}{3}&-\frac{1}{3}\\0&~\frac{2\sqrt{2}}{3}&-\frac{\sqrt{2}}{3}&-\frac{\sqrt{2}}{3}\\0&0&\frac{\sqrt{6}}{3}&-\frac{\sqrt{6}}{3}\end{array}\right];\quad n=3~,~\alpha =-\frac{1}{3}~\cdot
\end{equation}
Clearly the $2\times 3$ submatrix in the lower right corner of $S_3$ is a factor of $S_2$ So its column vectors are equiangular with the angle $-\pi/3$. Now assume that there are five equiangular vectors $s_{1}^{\prime},s_{2}^{\prime} ,\ldots,s_{5}^{\prime}$ in $\mathbb{R}^3$ so that $s_{i}^{\prime T}\cdot s_{j} ^{\prime}=\alpha^{'}$. These vectors can be rotated in such a way that $s^\prime _1$ is in direction to the vector $e_1$ and $s^\prime _2$ is in the plane spanned by $e_1,e_2$. So without loss of generality suppose that $s^\prime _{1},\ldots,s^\prime _{5}$ form a matrix column-wise as follows
\begin{equation}
S^\prime _3=\left[\begin{array}{ccccc}1&\alpha ^\prime & \alpha ^\prime &\alpha ^\prime &\alpha ^\prime \\0&*&*& *&*\\0&0&*& *&*\end{array}\right];\quad n=3~,~\alpha =\alpha ^\prime ,
\end{equation}
in which the column vectors are equiangular. So the $2\times 4$ submatrix in the lower right corner of $S^\prime _3$ has four equiangular column vectors in $\mathbb{R}^2$ which is a contradiction with $S_2$. By the same argument $n+1$ equiangular lines in $\mathbb{R}^n$ using $n$ equiangular lines in $\mathbb{R}^{n-1}$ can be constructed which form a $(n-1)\times n$ matrix as
\begin{equation}
S_{n-1}=\left[\begin{array}{ccccc} 1&-\frac{1}{n-1}&\ldots &-\frac{1}{n-1}&-\frac{1}{n-1}\\0&*&\ldots &*&*\\\vdots &\vdots &\ddots &\vdots &\vdots\\0&0&\ldots &*&*\end{array}\right];\quad\alpha =-\frac{1}{n-1}~\cdot
\end{equation}
Actually $S_n$ which is constructed by $n+1$ equiangular lines, contains a $(n-1)\times n$ submatrix in the lower right corner which is a factor of $S_{n-1}$. Therefore the $n\times (n+1)$ block matrix $S_n$ can be written as follows
\begin{equation}
S_n=\left[\begin{array}{cc}1&~-\frac{1}{n}e^T\\~&~\\0&~\rho S_{n-1} \end{array}\right];\quad \alpha =-\frac{1}{n}~,
\end{equation}
We must find the variable $\rho$ so that the column vectors of $S_n$ are normal. In the second column $[-\frac{1}{n},\rho,0,\ldots ,0]^T$ we have $\rho=\frac{\sqrt{n^2-1}}{n}$. It is notable that the column vectors of $S_n$ make a regular simplex which is a regular polytope. Therefore we see that the maximum number of ELs in $\mathbb{R}^n$ with the cosine of the angle $-\frac{1}{n}$ never exceed $n+1$.
\end{proof}
As noted before matrix $S$ obtained by deletion of the first column of $S_n$ is equiangular so that $S$ can be constructed from the theorem \ref{SDST} with $A=\text{diag}(\frac{1}{n},\frac{n+1}{n},\ldots,\frac{n+1}{n})$ and $\alpha=-\frac{1}{n}$.\\
On the other hand matrix $S_n$ can be extended to the factor of an orthogonal matrix of size $n+1$ by adding the row vector $(\sqrt{2}/2)e^T$ in the bottom. This new matrix can be written as $\sqrt{\tfrac{n+1}{n}}Q_{n}$ (condition $3^{'}$). So $S_n$ can be considered as the normalized projection of columns of $Q_n$ onto the orthogonal complement of $e_n= [0,0,\ldots,1]^T$. It implies that the row sum of $S_n$ is the zero vector, i.e., if $[S_n]_{ij}=s_{ij}\in \mathbb{R}^{n\times (n+1)}$, then $s_{ii}+(\sum _{j=i+1}^{n+1} s_{ij})=0$.
\begin{theorem}
Suppose that $s_1,s_2,\ldots,s_{n+1}$ indicate the equiangular vectors in $\mathbb{R}^n$ with the angle $\arccos(-1/n)$, Then the sequence $(s_i)_{i=1}^{n+1}$ is a tight frame with the lower and upper bounds equal to $(n+1)/n$.
\end{theorem}
\begin{proof}
From definition of ETF the result is obvious $(m=n+1)$. However we can prove it using Definition \ref{Deff}. So we show that
\begin{equation}
\frac{n+1}{n}\Vert x\Vert ^2=\sum _{i=1}^{n+1} (x^Ts_i)^2~,\quad \text{for all}~x\in\mathbb{R}^n.\label{frame}
\end{equation}
Without loss of generality let $x=[x_1,\ldots ,x_n]^{T}\in\mathbb{R}^n$ with $\Vert x\Vert =1$. For convenience suppose that $s_1,\ldots ,s_{n+1}$ are the column vectors of the corresponding matrix $S_n$ as noted in the Lemma \ref{Lmax}. The prove is by induction. For $n=1,2$ the result is trivial: $\sum_{i=1}^3\vert \langle x,s_i\rangle\vert^2=x_{1}^2+ (-\frac{1}{2}x_1+\frac{\sqrt{3}}{2}x_2)^2+(-\frac{1}{2}x_1-\frac {\sqrt{3}}{2}x_2)^2=\frac{3}{2}(x_1^2+x_2^2)=\frac{3}{2}$. For general $n$
\begin{align}
\sum _{i=1}^{n+1}\vert\langle x,s_i\rangle \vert^ 2&=x_1^2+(-\frac{1}{n}x_1+\rho s_{22}x_2)^2+ (-\frac{1}{n}x_1+\rho s_{23}x_2+ \rho s_{33}x_3)^2+\ldots\nonumber\\&+ (-\frac{1}{n}x_1+\rho s_{2n}x_2 +\ldots +\rho s_{nn}x_n)^2+ (-\frac{1}{n}x_1+\rho s_{2(n+1)} x_2+\ldots +\rho s_{n(n+1)}x_n)^2\nonumber\\&=(1+\frac{n}{n^2})x_1^2-\frac{2}{n} \rho x_1\big[(s_{22}+s_ {23}+\ldots +s_{2n}+s_{2(n +1)})x_2+(s_{33}+\ldots +s _{3(n+1)})x_3\nonumber\\&+\ldots +(s_{nn} +s_{n(n+1)})x_n\big]+\rho ^2\big[ (s_{22}x_2)^2+(s_{23}x_2 +s_{33}x_3)^2 +\ldots\nonumber\\&+(s_{2n}x_2 +\ldots +s_{nn}x_n)^2+ (s_{2(n+1)}x_2+\ldots +s_ {n(n+1)}x_n)^2\big]\nonumber\\&=\frac{n+1}{n}x_1^2-0+\rho ^2\frac{n}{n-1}(x_2^2+x_3^2+\ldots +x_n^2)=\frac{n+1}{n}x_1^2+\rho ^2\frac{n}{n-1}(1-x_1^2) \nonumber\\&=\frac{n^2-1}{n^2}\cdot\frac{n}{n-1}=\frac{n+1}{n}.
\end{align}
\end{proof}
Notice that EA is not designed for angles greater than $\pi/2$. To solve this problem for an angle $\tfrac{\pi}{2}<\Theta <\pi$ where $\theta =\pi -\Theta$, the equation $v_2=q_2+\cot\theta ~s_1$ in \cite{Rivaz} (eq. (2.1)), where $s_2=v_2/\Vert v_2\Vert$, holds. Because $v_2 =q_2+\cot\Theta ~s_1 =q_2-\cot\theta~s_1$, so the new vector $\hat{v}_2$ is the symmetric of the initial $\hat{v}_2$ with respect to the vector $v$ as the axis of symmetry. Also this new equation can be written as $v_2 =q_2-\vert\cot \Theta\vert s_1$. This method can be applied for every $v_k~ (k\geq 3)$. So it is enough to change the plus sign in the line 5 of EA to the minus as follows
\begin{equation}
v_{k}=q_k-\left(\sqrt{k-1}/\sqrt{\left ((\sec\Theta-1)(\sec\Theta +k-1)\right)}\right)s~. \label{Theta}
\end{equation}
Moreover, \eqref{Theta} determines the bound of $\Theta$ to construct a set of 'linearly independent' equiangular vectors from a set of $k$ linearly independent vectors. Since $\sec\Theta<-1$, then $\sec\Theta<1-k$. It implies 
\begin{equation}
\cos\Theta>-\tfrac{1}{k-1}. \label{Ctheta}
\end{equation}
The inequality \eqref{Ctheta} shows that the infimum of $\Theta$ for which any set of equiangular vectors of size $k$ in $\mathbb{R}^n$ be linearly independent, is $\cos ^{-1} (-\tfrac{1}{k-1})$. Otherwise if the lower bound holds the outcome of EA is the same as Lemma \ref{Lmax}.
\section{Conclusion}
In this paper we introduced an algorithm thereby any set of linearly independent vectors $v_1,\ldots,v_n$ in $\mathbb{R}^n$ is ordered as a set of equiangular vectors with an arbitrary angle between $0$ and $\arccos(\frac{-1}{n-1})$. The outcome vectors of this algorithm produce a matrix whose column vectors are equiangular, is called equiangular matrix. We can say that working with the equiangular	 matrices is straightforward because their column vectors are equiangular with unit norm. These matrices can be appeared in some matrix factorizations.

\end{document}